\documentclass[letterpaper, 10 pt, conference]{ieeeconf} 

\usepackage{kdayiieee}
\usepackage{mathrsfs}
\usepackage{xcolor}

\IEEEoverridecommandlockouts                              

\title{\LARGE \bf
Projected Push-Pull For Distributed Constrained Optimization Over
Time-Varying Directed Graphs (extended version)
}

\author{Orhan Eren Akg\"un$^{*}$, Arif Kerem Day\i$^{*}$, Stephanie Gil, and Angelia Nedi\'c
\thanks{(*Co-primary authors). O.~E.~Akgün, A.~K.~Day\i, and S.~Gil are with the School of Engineering and Applied Sciences, Harvard University, USA: {\small erenakgun@g.harvard.edu, keremdayi@college.harvard.edu, sgil@seas.harvard.edu}. A.~Nedi\`c is with the School of Electrical, Computer and Energy Engineering, Arizona State University, USA: {\small Angelia.Nedich@asu.edu}.}%
\thanks{The authors gratefully acknowledge partial support through NSF CNS 2147641 and 2147694.}%
}

\begin{document}

\maketitle
\thispagestyle{empty}
\pagestyle{empty}

\begin{abstract}
We introduce the Projected Push-Pull algorithm that enables multiple agents to solve a distributed constrained optimization problem with private cost functions and global constraints, in a collaborative manner. 
Our algorithm employs projected gradient descent to deal with constraints and a lazy update rule to control the trade-off between the consensus and optimization steps in the protocol. We prove that our algorithm achieves geometric convergence over time-varying directed graphs while ensuring that the decision variable always stays within the constraint set. We derive explicit bounds for step sizes that guarantee geometric convergence based on the strong-convexity and smoothness of cost functions, and graph properties. Moreover, we provide additional theoretical results on the usefulness of lazy updates, revealing the challenges in the analysis of any gradient tracking method that uses projection operators in a distributed constrained optimization setting. We validate our theoretical results with numerical studies over different graph types, showing that our algorithm achieves geometric convergence empirically.

\end{abstract}

\section{INTRODUCTION}
In this paper, we are concerned with a class of distributed optimization problems where a set of $n$ agents are trying to solve a problem with the structure:
\begin{align}
    \min_{x\in \mX} f(x),\ \text{ where }\ f(x)\triangleq\frac{1}{n}\sum_{i=1}^n f_i(x),
    \label{eq:problem_def_intro}
\end{align}
where $x$ is the decision variable, each cost function $f_i:~\R^d \to \R$ is known by agent $i$ only and is strongly convex with Lipschitz continuous gradients, and the constraint set $\mX \subseteq \R^d$ is closed and convex. We are interested in the case where agents communicate over a possibly time-varying directed graph $\mG_k=(\mV,\mE_k)$ where $\mV$ with $|\mV|=n$ represents the set of agents and the set $\mE_k$ represents the directed communication links at time $k$. This setup has various applications in control \cite{nedic_control_2018}, robotics \cite{carnevale_coordination}, and sensor networks \cite{nedic_distributed_sensors}. 

Many distributed optimization applications demand fast algorithms due to time and computational constraints, which makes the convergence rate of the used algorithm critical. However, a simple extension of gradient descent to distributed optimization does not achieve geometric convergence even with strongly convex cost functions \cite{harnessing_smoothness_cdc,harnessing_smoothness_journal}. 
Therefore, gradient tracking was introduced to achieve geometric convergence in undirected \cite{harnessing_smoothness_cdc, xu2015augmented, DIGing} and directed graphs \cite{DIGing, FROST_og, AB_usman, push-pull}. In gradient tracking methods, agents maintain a decision variable and an estimate of the global gradient. At each step, agents first perform a consensus step and an optimization step on the decision variable using the estimated global gradient. Then, they update their estimate of their global gradient using their neighbors' estimates and their local gradient.
In particular, the Push-Pull algorithm introduced in \cite{AB_usman, push-pull} achieves geometric convergence in directed, time-varying graphs \cite{AB_usman_time_varying, nedich2022ab-timevarying}. Unlike other gradient tracking methods, Push-Pull uses row and column stochastic mixing matrices for averaging decision and gradient tracking variables, respectively. Therefore, it does not require estimating the non-one Perron vector of the mixing matrix, which would introduce additional communication and computation costs. 
However, it does not handle constrained optimization problems. Indeed, despite great progress in distributed unconstrained optimization algorithms, their counterparts in the constrained optimization space still remain underexplored. Our goal in this work is to develop a projected gradient descent based Push-Pull algorithm variant to achieve geometric convergence in \emph{constrained optimization problems} over time-varying directed graphs.

Extending gradient tracking methods, including Push-Pull, to handle constrained optimization is a non-trivial task. Projection based algorithms for constrained optimization have some fundamental differences from their counterparts for unconstrained optimization.
First, the non-linearity of the projection operator limits our ability to manipulate the mixing matrices in the analysis, which is an essential part of the analysis in the unconstrained case. 
Second, in the unconstrained case, the global gradient vanishes at the optimal point, which is heavily used in existing analyses. However, the gradient does not necessarily vanish at the optimum in the constrained setting. 
Since the gradient at the optimal point can be non-zero, the step size in the constrained case does not give fine-grained control over the tradeoff between different errors, such as the optimality error, consensus error, and the gradient tracking error, which are standard in the analysis of all gradient tracking methods.

The work \cite{liu2020discrete} is an initial attempt to overcome these challenges in gradient tracking methods in the constrained setting. The reference \cite{luan2023distributed} improves this algorithm by removing the need for multiple projection calculations at every step. Both these works use only row stochastic mixing matrices and, therefore, require estimating the non-one Perron vector of the mixing matrix, increasing computation and communication costs. 
Moreover, both algorithms are proposed for static graphs, and they require multiple consensus steps per optimization step, further increasing communication costs. Recent work in \cite{error_projection} eliminates the need for multiple consensus steps. However, their results are limited to static and undirected graphs, and the decision variables are not guaranteed to stay in the constraint set at every time step. 

Ideally, we want a distributed constrained optimization algorithm that 1)~achieves a geometric convergence rate, 2)~works for directed graphs and time-varying graphs, 3)~has low communication cost (i.e., does not require multiple consensus steps or estimation of additional system parameters), 4)~minimizes the number of costly operations such as projection, and 5)~keeps the decision variable in the constraint set at all time steps.
With this motivation, we introduce the Projected Push-Pull algorithm that satisfies all the aforementioned requirements. Similar to Push-Pull, we employ row and column stochastic mixing matrices. This allows our algorithm to work in directed time-varying graphs without needing to estimate the non-one Perron vector of the mixing matrices. To handle the constrained case, we use projection to keep the decision variables in the constraint set and an extra step size to control the tradeoff between consensus and optimization. We prove the geometric convergence rate of the algorithm for time-varying directed graphs. Our contributions can be summarized as follows
\begin{itemize}
    \item We introduce a novel distributed projected gradient algorithm based on the Push-Pull to solve distributed constrained optimization problems with structure as in~\Cref{eq:problem_def_intro} over time-varying directed graphs.
    \item We prove that with a small enough step size, our algorithm has a geometric convergence rate for time-varying directed graphs. We characterize the valid range for the step size based on various problem-based parameters, such as the smoothness and strong convexity of the cost functions and properties of the communication graph.
    \item We provide impossibility results that show some fundamental limitations of distributed gradient methods using projection in constrained optimization settings. 
    \item We empirically show that our algorithm attains geometric convergence via numerical studies with different graph types.
\end{itemize}

\section{Notation \& Terminology}
All vectors are column vectors by default unless stated otherwise. The $i$-th entry of a vector $u$ is denoted by $u_i$; it is $[u_k]_{i}$ if $u_k$ is time varying where $k \geq 0$ is the time step. For a vector $u$, $\min u$ and $\max u$ denote the smallest and largest entries of $u$, respectively. For any matrix $A$, we denote its $ij$-th entry by $A_{ij}.$ If it is a time-varying matrix, we denote it by $[A_k]_{ij}.$ We denote the smallest positive element of a non-negative matrix $A$ by $\min \{A^+\}.$ A non-negative matrix is row-stochastic if all of its row sums are equal to $1$, and it is column stochastic if all of its column sums are equal to~$1.$

We use $\langle a, b \rangle$ to denote the Euclidean inner product and $\norm{x}=\sqrt{\langle x, x \rangle}$ to denote the Euclidean norm. For any vector $u \in \R^n$ with $u_i > 0$, we define the $u$-weighted norm of $\boldx \in \R^d \times \dots \times \R^d$ ($n$ copies of $\R^d$) as $\norm{\boldx}_{u}=\sqrt{\sum_{i=1}^n u_i \norm{x_i}^2}$ where $x_i \in \R^d$. 

A directed graph $\mG=(\mV, \mE)$ is said to be strongly connected if there is a directed path between any pair of the nodes in the graph. Finally, we define a projection operator as follows:
\begin{definition}[Projection onto $\mX$]
    Let $\mX\subseteq \R^d$ be closed and convex and $\norm{\cdot}$ be the Euclidean norm on $\R^d$. Then, we define the projection operator $\proj{\cdot}: \R^d \to \R^d$ as follows
    \begin{align*}
        \proj{x} = \arg\min_{z\in \mX} \norm{x-z}.
    \end{align*}
\end{definition}

\section{PROBLEM SETUP}
We consider a distributed multi-agent system of $n$ agents where agents need to solve a distributed optimization task in a collaborative manner. The agents' goal is to solve the following constrained minimization problem 
\begin{equation}
    \min_{x\in \mX} f(x),\ \text{ where }\ f(x)\triangleq\frac{1}{n}\sum_{i=1}^n f_i(x),
    \label{eq:minimization_goal}
\end{equation}
where each cost function $f_i: \R^d \to \R$ is known by agent $i$ only and the constraint set $\mX \subseteq \R^d$ is closed and convex. We make the following assumptions about the cost functions:
\begin{assumption}[Strongly Convex  Objective]\label{assumption:objective-strong-convexity}
    For all agents $i$, $f_i(x)$ is $\mu$-strongly convex, i.e, for some $\mu >0$, we have $\langle \nabla f_i(x)-\nabla f_i(y), x- y\rangle \geq \mu \normsq{x-y}$, for all $x,y \in \R^d.$ 
\end{assumption}

\begin{assumption}[Lipschitz Continuity of Gradients]\label{assumption:objective-smoothness}
    For all agents $i$, $\nabla f_i(x)$ is $L$-Lipschitz continuous, i.e, for some $L >0$, $\norm{\nabla f_i(x)-\nabla f_i(y)}\leq L \norm{x-y}$, for all $x,y \in \R^d.$ 
\end{assumption}

We assume that at each time step $k \in \N$, agents communicate over a directed graph, denoted by $\mG_k=(\mV,\mE_k)$. The set $\mV$ with $|\mV|=n$ represents the set of agents and the set $\mE_k$ represents the directed communication links at time $k$. An edge $(i,j)\in\mE_k$ indicates that agent $i$ can send information to agent $j$ at time $k$. Moreover, if $(i,j)\in\mE_k$, we say that $i$ is an in-neighbor of $j$ and $j$ is an out-neighbor of $i$.  
We make the following assumption on the communication graphs $\mG_k.$
\begin{assumption}[Strong Connectivity]
    \label{assumption:graph-connectivity}
    $\mG_k$ is strongly connected for all $k$.
\end{assumption}

\section{ALGORITHM}
In this section, we introduce the Projected Push-Pull algorithm. In the algorithm, all agents maintain two decision variables $x_i[k]$ and $z_i[k]$, and a gradient tracking variable $y_i[k]$. Agents initialize $x_i[0]=z_i[0]\in \mX$ arbitrarily, and $y_i[0]=\nabla f_i(x_i[0])$. At each communication round $k$, agents get $z_j[k]$ and the scaled gradient tracking variable $[C_k]_{ij}y_j[k]$ from their in-neighbors and perform the following updates:
\begin{subequations}\label{eq:update_rule}
    \begin{align}
    x_{i}[k+1]&= \sum_{j=1}^n [R_{k}]_{ij}z_{j}[k],\label{eq:update_rule_a}\\
    y_{i}[k+1]&= \sum_{j=1}^n [C_{k}]_{ij}y_j[k]+\nabla f_i(x_i[k+1])-\nabla f_i(x_i[k]),\label{eq:update_rule_b}\\
    z_i[k+1]&=(1-\lambda)x_i[k+1]\label{eq:update_rule_c}\\
    &+\lambda \proj{x_i[k+1]-\eta y_i[k+1]},\nonumber
\end{align}
\end{subequations}
where $\eta>0, \lambda \in (0, 1]$ are two different step sizes. We will refer to \Cref{eq:update_rule_c} as the lazy update rule. We formally describe how agent $i\in \mV$ runs the protocol in \Cref{alg:push-pull}. 
\begin{algorithm}
\caption{Projected Push-Pull}
\begin{algorithmic}[1]
\REQUIRE Optimization parameters $\eta, \lambda$, chosen according to \Cref{theorem:convergence-of-protocol}.
\STATE Each agent $i$ simultaneously does the following:
\STATE Initialize $x_i[0]=z_i[0] \in \mX$ arbitrarily, and set $y_i[0]=~ \nabla f_i(x_i[0])$. 
\WHILE{$k=0, 1, \dots$}
\STATE Determine coefficients $[C_{k}]_{ji}$ for $j\in \mV.$
\STATE Send $z_i[k]$, $[C_k]_{ji}y_i[k]$ to out-neighbors $j \in \mN_{i}^{out}[k].$
\STATE Receive $z_j[k]$, $[C_k]_{ij}y_j[k]$ from in-neighbors $j \in ~ \mN_{i}^{in}[k].$
\STATE Determine coefficients $[R_k]_{ij}$ for $j\in \mV.$
\STATE Perform the consensus update using \Cref{eq:update_rule_a}:
$$x_i[k+1] \gets \sum_{j=1}^n [R_k]_{ij}z_j[k].$$
\STATE Perform the gradient tracking update using \Cref{eq:update_rule_b}:
\begin{align*}
    y_i[k+1] \gets &\sum_{j=1}^n [C_k]_{ij}y_j[k] \\ &+ \nabla f_i(x_i[k+1]) - \nabla f_i(x_i[k]).
\end{align*}
\STATE Perform the lazy optimization update using \Cref{eq:update_rule_c}:
\begin{align*}
    z_i[k+1] \gets &(1-\lambda) x_i[k+1] \\
                    &+  \lambda \proj{x_i[k+1] - \eta y_i[k+1]}.
\end{align*}

\ENDWHILE
\end{algorithmic}
\label{alg:push-pull}
\end{algorithm}

We make the following assumptions on the matrices $R_k$ and $C_k$: 
\begin{assumption}[Graph Compatibility of $R_k$]
    \label{assumption:compatible_R}
    For all $k>0$, the matrix $R_k$ is row stochastic and it is compatible with the graph $\mG_k$, i.e., $[R_k]_{ij}>0$ if and only if $(j,i)\in \mE_k$ or $i=j$ and $[R_k]_{ij}=0$ otherwise. Moreover, for some $\Rmin> 0$ we have $\min \{R_k^+\} \geq \Rmin$ for all $k > 0.$ 
\end{assumption}

\begin{assumption}[Graph Compatibility of $C_k$]
    \label{assumption:compatible_C}
    For all $k>0$, the matrix the matrix $C_k$ is column stochastic and it is compatible with the graph $\mG_k$, i.e., $[C_k]_{ij}>0$ if and only if $(j,i)\in \mE_k$ or $i=j$ and $[C_k]_{ij}=0.$ Moreover, for some $\Cmin> 0$ we have $\min \{C_k^+\} \geq \Cmin$ for all $k > 0.$ 
\end{assumption}
Under Assumption~\ref{assumption:compatible_C}, the algorithm satisfies the gradient tracking property, that is $
       \sum_{i=1}^n y_i[k] = \sum_{i=1}^n \nabla f_i(x_i[k]),
    $ at each time step $k.$
    
The key differences between this algorithm and the AB/Push-Pull algorithm \cite{nedich2022ab-timevarying} are as follows: 1) agents calculate the gradients at $x_i[k]$, which is after the consensus step \Cref{eq:update_rule_a}, 2) we introduce a projection operator in the calculation of $z_i[k]$, and 3) we use an additional step size $\lambda$ to give agents more control over the trade-off between the consensus and optimization. 

\section{MAIN RESULTS}
In this section, we state the main results concerning the convergence of our algorithm to the optimal point. First, will provide some core results about the behavior of graph compatible row stochastic and column stochastic matrices, and their contraction behavior. Then, we provide our main theorem showing the geometric convergence of our algorithm. We note that the analysis accompanying these results is given in the next section \Cref{sec:analysis}. 
\subsection{Preliminaries} 
We use the following lemmas to define stochastic vectors that will be used in our analysis.
\begin{lemma}[\cite{nguyen2022distributed}, Lemma 5.4 and \cite{nedich2022ab-timevarying}, Lemma 3.3]
\label{lemma:phi-sequence}
    Let \Cref{assumption:graph-connectivity} hold and $\{R_k\}$ be a row stochastic matrix sequence satisfying \Cref{assumption:compatible_R}. Then, there exists a sequence of positive stochastic vectors $\{\phi_k\}$ such that $\phi_{k+1}^\intercal R_k = \phi_{k}^\intercal$, where the entries of each $\phi_k$ are positive and have the uniform lower bound
$[\phi_k]_i \geq \frac{(\Rmin)^n}{n}$ for all $i \in [n]$ and $k \geq 0$.  
\end{lemma}
\begin{lemma}[ \cite{nedich2022ab-timevarying}, Lemma 3.4]
\label{lemma:pi-sequence}
    Let \Cref{assumption:graph-connectivity} hold and $\{C_k\}$ be a matrix sequence satisfying \Cref{assumption:compatible_C}. Set $\pi_0 = \frac{1}{n}\boldone$ and define the sequence $\pi_{k+1} = C_k \pi_k $. Then, each vector in $\{\pi_k\}$ is stochastic, and we have $[\pi_k]_i \geq \frac{(\Cmin)^n}{n}.$
\end{lemma}
Now, we will define two lemmas about contractions of matrices $R_k$ and $C_k$ which allow the consensus of $x_i[k]$ and $y_i[k]$ values. 
\begin{lemma}[\cite{nguyen2022distributed}, Lemma 6.1]
    \label{lemma:R-contraction}
    Let $\mG=(\mV,\mE)$ be a strongly connected graph, and the row stochastic matrix $R$ be compatible with the graph. Let $\phi$ be a stochastic vector and $\phi'$ be a non-negative vector such that $\phi'^{\intercal}R=\phi^\intercal.$ Consider the vectors $z_1, z_2, \dots, z_n \in \R^d$ and $x_i = \sum_{j=1}^n R_{ij} z_j$ for all $i\in \mV$. Also define $\hat z_\phi \triangleq \sum_{i=1}^n \phi_i z_i$. Then, we have
    \begin{align*}
        \sum_{i=1}^n \phi'_i \normsq{x_i - u} \leq \sum_{j=1}^n \phi_j \normsq{z_j-u}\\
        -\frac{\min(\phi')(\min (R^+))^2}{\max^2 (\phi) \mathsf D(\mG) \mathsf K(\mG)}\sum_{j=1}^n \phi_j\normsq{z_j-\hat z_{\phi}},
    \end{align*}
    where $\mathsf D(\mG)$ and $\mathsf K(\mG)$ are the diameter and the maximum edge utility of $\mG$, respectively, as in \cite[Lemma 6.1]{nguyen2022distributed}. Letting $u= \sum_{j=1}^n \phi_j z_j=\sum_{j=1}^n \phi_j' x_j\triangleq \hat x_{\phi'}$, we get
    \begin{align*}
        \sqrt{\sum_{i=1}^n \phi_i' \normsq{x_i- \hat x_{\phi'}}} \leq \sigma \sqrt{\sum_{i=1}^n \phi_i \normsq{z_i - \hat z_{\phi}}},
    \end{align*}
    where $\sigma = \sqrt{1-\frac{\min(\phi') (\min R^+)^2}{\max^2 (\phi) \mathsf D(\mG) \mathsf K(\mG)}}\in (0, 1)$.
\end{lemma}
\begin{lemma}[\cite{nedich2022ab-timevarying}, Lemma 4.5]
    \label{lemma:C-contraction}
    Let $\mG=(\mV,\mE)$ be a strongly-connected graph and $C$ be a column stochastic matrix compatible with $\mG$. Assume $y_1, y_2, \dots, y_n \in \R^d$ and $v_i = \sum_{j=1}^n C_{ij} y_j$ for all $i\in \mV$. Let $\pi \in \R^n$ be a positive stochastic vector and $\pi' = C\pi$. Then, we have 
    \begin{align*}
        \sqrt{\sum_{i=1}^n \pi_i' \normsq{\frac{v_i}{\pi_i'}-\sum_{l=1}^n y_l}} \leq \tau \sqrt{\sum_{i=1}^n \pi_i \normsq{\frac{y_i}{\pi_i}-\sum_{l=1}^n y_l}},
    \end{align*}
    where $\tau = \sqrt{1- \frac{\min^2 (\pi) (\min C^+)^2}{\max^2(\pi) \max(\pi')\mathsf{D}(\mG) \mathsf K(\mG)}} \in (0,1).$
\end{lemma}

Lastly, we introduce a lemma showing  the contraction properties of the projected gradient method. This lemma is an adaptation of a standard result in optimization for the projected gradient method (see \cite[Lemma 10]{harnessing_smoothness_journal}).
\begin{lemma}[Projected Gradient Contraction]
Let $\mX \subseteq \R^d$ be closed and convex set, and 
let $f: \R^d \to \R$ be $\mu$-strongly convex and $L$-smooth. Define $\mT_\eta(x)= \proj{x -\eta \nabla f(x)}$. For $0 < \eta < \frac{2}{\mu + L}$, we have
    \begin{align*}
        \norm{\mT_\eta (x) - \mT_\eta(y)} \leq q(\eta) \norm{x-y}
    \end{align*}
    where $q(\eta)=  1-\eta \mu < 1$.
\label{lemma:projected-gd-contraction}
\end{lemma}

\subsection{Convergence Results}
 The convergence of \Cref{alg:push-pull} will be determined entirely by 3 critical error terms, or distances: 1) agents' decision variables' distances to the optimal point, 2) the consensus of the decision variables, and 3) the convergence of gradient tracking variables. We define these respective  error terms mathematically as follows:
\begin{align}
        \label{eq:optimality-error}
        \normphik{\boldx[k]-\boldx^*}\triangleq\sqrt{\sum_{i=1}^n \phiki \normsq{x_i[k]-x^*}},
\end{align}
where $\boldx[k]=(x_1[k], \dots, x_n[k]),$ $\boldx^*=(x^*, \dots, x^*),$ and the vectors $\phi_k$ satisfy \Cref{lemma:phi-sequence}. 
\begin{equation}
    D(\boldx[k], \phi_k)\triangleq\sqrt{\sum_{j=1}^n\sum_{i=1}^n\phiki\phikj \norm{x_i[k]-x_j[k]}^2},
    \label{eq:consensus-error} 
\end{equation}
\begin{align}
    S(\boldy[k], \pi_k)\triangleq \sqrt{\sum_{i=1}^n [\pi_k]_i \norm{\frac{y_i[k]}{[\pi_k]_i}-\sum_{l=1}^n y_l[k]}^2},
    \label{eq:gradient-tracking-error}
\end{align}
where $\boldy[k]=(y_1[k], \dots, y_n[k])$ and $\pi_k$ satisfy \Cref{lemma:pi-sequence}. We call the term $\normphik{\boldx[k]-\boldx^*}$ the optimality gap, $D(\boldx[k], \phi_k)$ the consensus error, and $(\boldy[k], \pi_k)$ the gradient tracking error. Now, we combine the errors in a single vector as $\bolde[k]=(\normphik{\boldx[k] - \boldx^*}, D(\boldx[k], \phi_k), S(\boldy[k], \pi_k))^\intercal.$    
We aim to show that $\lim_{k\to \infty}\bolde[k]=0$ with a geometric rate. Hence, we want to find some matrix $M(\eta, \lambda)$ with spectral radius $\rho(M(\eta, \lambda)) <1$ such that $\bolde[k+1] \leq M(\eta, \lambda) \bolde[k]$. This will give us the desired geometric rate. With this motivation, we now give the composite relation between the error terms at step $k+1$ and error terms at step $k$. First, define 
\begin{align*}
    \sigma_k &\triangleq \sqrt{1-\frac{\min(\phi_{k+1}) (\min R_k^+)^2}{\max^2 (\phi_k) \mathsf D(\mG_k) \mathsf K(\mG_k)}}\in (0,1),\\
    \tau_k &\triangleq \sqrt{1- \frac{\min^2 (\pi_{k}) (\min C_k^+)^2}{\max^2(\pi_{k}) \max(\pi_{k+1})\mathsf{D}(\mG_k) \mathsf K(\mG_k)}}\in (0,1).
\end{align*}
which are the coefficents of contraction due to $R_k$ and $C_k$ respectively (as defined in \Cref{lemma:R-contraction} and \Cref{lemma:C-contraction}), at time $k$. Notice that $\sigma_k, \tau_k$ are uniformly bounded above by constants less than $1$ due to Assumption~\ref{assumption:compatible_R} and Assumption~\ref{assumption:compatible_C}, and  \Cref{lemma:phi-sequence} and \Cref{lemma:pi-sequence}. Then, also define
\begin{align*}
    r_k \triangleq \sqrt{\frac{1}{\min \pi_k}} + \sqrt{n}, \quad
    \varphi_k \triangleq\sqrt{\frac{1}{\min \phi_k}}.
\end{align*}
Notice that since the entries of $\pi_k$ and $\phi_k$ are bounded above and below uniformly across time, the min and max elements are also bounded uniformly over time. Therefore, we can define $r \triangleq \sup_{k \geq 0} r_k , \varphi \triangleq \sup_{k\geq 0} \varphi_k,  \psi \triangleq \inf_{k \geq 0} \min \pi_k > 0, \sigma \triangleq \sup_{k \geq 0} \sigma_k <1 , \tau \triangleq \sup_{k \geq 0} \tau_k < 1.$ Then, we have the following proposition describing the evolution of the errors.
\begin{proposition}[Composite Relation]
    Let~\Crefrange{assumption:objective-strong-convexity}{assumption:compatible_C} hold and let $\eta < \frac{1}{L} \leq \frac{1}{Ln \psi}$. Then, we have
    \begin{align}
        \bolde[k+1] \leq M(\eta, \lambda) \bolde[k],
        \label{eq:composite-error-relation}
    \end{align}
    where the inequality is elementwise and  $M(\eta, \lambda)$ is equal to
    \begin{align*}
    \begin{bmatrix}
            1-\eta \lambda n \psi \mu & \lambda \varphi \sqrt{n} & \lambda  L^{-1}  \\
            2\lambda & \sigma + 2\lambda \sqrt{n} \varphi & 2 \lambda  L^{-1} \\
            2 \lambda Lr \varphi & Lr\varphi (1+\sigma) + \lambda L r \varphi^2 \sqrt{n} & \tau + \lambda r \varphi 
        \end{bmatrix}.
\end{align*}
    \label{proposition:composite-relation}
\end{proposition}

\begin{theorem}[Convergence]
\label{theorem:convergence-of-protocol}
    Let~\Crefrange{assumption:objective-strong-convexity}{assumption:compatible_C} hold. Let $0<\eta < \frac{1}{nL}$ and \begin{align*}
    \lambda < \min \left\{\frac{1-\sigma}{2\varphi \sqrt{n}}, \frac{1-\tau}{r\varphi}, \frac{\eta n \psi \mu (1-\sigma)(1-\tau)}{K}\right\},
\end{align*} where
\begin{align*}
    K=
    (1+\eta n \psi \mu) \varphi [2\sqrt{n}(1-\tau)+r(1-\sigma) +2r (1+\sigma)].
\end{align*}
Then, 
\begin{align*}
    \lim_{k\to \infty} \norm{x_i[k]-x^*}=0 \; \text{ for all } i\in\mV,
\end{align*}
where $x^*$ is the solution to problem~\eqref{eq:minimization_goal}.
    Moreover, the convergence rate is geometric with rate $\rho(M(\eta, \lambda))<1,$ where $\rho(\cdot)$ denotes the spectral radius of a matrix.
\end{theorem}
The proof of \Cref{theorem:convergence-of-protocol} is given in \Cref{sec:proof_main_thm}. The proof shows that by choosing $\lambda$ in the specified range, we can make the diagonals of $M$ less than $1$ and $\det (M(\eta, \lambda) - I) < 0$, which are sufficient to show $\rho(M(\eta, \lambda)) < 1$.
\section{ANALYSIS}
\label{sec:analysis}
In this section, at first, we provide all the necessary results for the proof of \Cref{theorem:convergence-of-protocol}. Then, we provide two impossibility results providing insights into our algorithm design and the analysis. Some of the proofs in this section are given in \Cref{sec:appendix}.
\subsection{Bounding Optimality Gap}
\label{sec:bounding_optimality_gap}
We start the analysis of the optimality gap under our algorithm. First, notice that we have $\normphikp{\boldx[k+1]-\boldx^*} \leq \normphik{\boldz[k]-\boldx^*}$ from \Cref{lemma:R-contraction} with $u=x^*$. 
 Hence, we will focus on the analysis of $\normphik{\boldz[k]-\boldx^*}.$ Our strategy is to split the error into two cases: the error we would have if agents had the perfect gradient knowledge and the error coming from the gradient tracking. To represent the case where agents have the perfect gradient knowledge, we define
\begin{align*}
    w_i[k]=(1-\lambda)x_i[k]-\lambda\proj{x_i-\eta n[\pi_k]_i \nabla f(x_i[k])},
\end{align*}
for each agent $i$ where $\nabla f(x_i[k])\triangleq \frac{1}{n}\sum_{l=1}^n \nabla f_l(x_i[k])$ and stack these vectors in the matrix $\boldw[k]$. With this definition, we have $\normphik{\boldz[k]-\boldx^*} \leq \normphik{\boldz[k]-\boldw[k]} +\normphik{\boldw[k]-\boldx^*}$ by the triangular inequality. We establish a bound on the first term with the following lemma.
\begin{proposition}[Bounding Error from Imperfect Gradients]  \label{proposition:imperfect_gradient_error}
    Let \Crefrange{assumption:objective-smoothness}{assumption:compatible_C} hold. Then, we have for all $k\ge0$,
    \begin{align*}
        \normphik{\boldz[k]-\boldw[k]} &\leq \eta \lambda L \varphi_k \sqrt{n} \Dxk \\
        &+\eta \lambda \Syk.
    \end{align*}
\end{proposition}
Next, we define the following terms to capture the contraction due to the lazy update rule \Cref{eq:update_rule_c}:
\begin{align}
    q(\eta, \lambda)= 1-\lambda + \lambda q(\eta), \text{ and} \\
    q_k(\eta, \lambda) = \max_i q(\eta n\piki, \lambda),
\end{align}
where $q(\eta)$ is the contraction we have in the projected gradient method as defined in \Cref{lemma:projected-gd-contraction}. Now, we can derive our main result of the optimality gap:
\begin{lemma}[Optimality Gap Bound]\label{lemma:optimality-gap}
    Let \Crefrange{assumption:objective-strong-convexity}{assumption:compatible_C} hold. Let $\eta< \frac{1}{nL}$ and $\lambda \in (0, 1]$. Then, we have for all $k\ge0$,
    \begin{align*}
        \normphikp{\boldx[k+1]-\boldx^*} \leq q_{k}(\eta, \lambda) \normphik{\boldx[k]-\boldx^*} \\
        +\eta \lambda L \varphi_k \sqrt{n}D(\boldx[k], \phi_k)+\eta \lambda  S(\boldy[k], \pi_k).
    \end{align*}
\end{lemma}
This lemma shows that we can control the error contributions coming from the consensus and gradient tracking errors by choosing smaller step sizes $\lambda$ or $\eta$.

\subsection{Bounding Consensus Error}
Similar to the analysis of the optimality error in the previous section, we want to isolate the gradient tracking error. Let $\boldu \in \R^d \times \dots \times \R^d$ ($n$ copies of $\R^d$) and $a \in \R^n$ be a positive stochastic vector. Then, similar to the consensus error $\Dxk$, we can define
\begin{align}
    D(\boldu, a)\triangleq \sqrt{\sum_{i=1}^n \sum_{j=1}^n a_i a_j \normsq{u_i - u_j}}. \label{eq:D-error-definition}
\end{align}
Hence, in light of \Cref{lemma:R-contraction}, notice that $\Dxk \leq \sigma_k D(\boldz[k], \phi_k)$. Then, we isolate the gradient tracking error contained in $D(\boldz[k], \phi_k)$ with the following proposition:
\begin{proposition}[Isolating Gradient Tracking Error]
Let \Crefrange{assumption:graph-connectivity}{assumption:compatible_R} hold. Then, we have for all $k\ge0$,
\begin{align*}
    D(\boldz[k], \phi_k) \leq 2\normphik{\boldz[k]-\boldw[k]} + D(\boldw[k], 
    \phi_k).
\end{align*}
\label{proposition:removing-gradient-error-consensus}
\end{proposition}
We already have a bound on the term $\normphik{\boldz[k]-\boldw[k]}$ from \Cref{proposition:imperfect_gradient_error}. Therefore, we can complete the consensus error analysis by analyzing consensus under global gradient knowledge, which is captured by the term $D(\boldw[k], \phi_k)$.
\begin{proposition}\label{proposition:w-consensus-bound}
Let~\Crefrange{assumption:objective-strong-convexity}{assumption:compatible_C} hold. Let $\eta<~\frac{1}{nL}$ and $\lambda \in (0, 1]$. Then, we have for all $k\ge0$,
\begin{align*}
    D(\boldw[k], \phi_k)&\leq q_{k}(\eta, \lambda) \Dxk \\
    &+ 2 \lambda q_k(\eta, 1) \normphik{\boldx[k]-\boldx^*}.
\end{align*}
\end{proposition}
Now, we can combine \crefrange{proposition:imperfect_gradient_error}{proposition:w-consensus-bound} to obtain the final bound for $\Dxkp$.
\begin{lemma}[Consensus Error Bound] 
    \label{lemma:consensus-error}
Let~\Crefrange{assumption:objective-strong-convexity}{assumption:compatible_C} hold. Let $\eta< \frac{1}{nL}$ and $\lambda \in (0, 1]$. Then, we have for $k\ge0$,
    \begin{align*}
        \Dxkp \leq 2 \lambda\sigma_k q_k(\eta, 1) &\normphik{\boldx[k]-\boldx^*} \\+ (\sigma_kq_{k}(\eta, \lambda) + 2\eta \lambda \sigma_k L  \varphi_k\sqrt{n}) &\Dxk \\ + 2\eta \lambda \sigma_k &\Syk.
    \end{align*}
\end{lemma}
The error contribution from the optimality gap and gradient tracking error can be made small by choosing a small step size $\lambda$. Moreover, the contribution from the consensus error in previous step comes with a contraction coefficient $\sigma_k$ and some additional error which can be made small with $\lambda$.

\subsection{Bounding Gradient Tracking Error}
In this section, we analyze the gradient tracking error $\Sykp$. Recall that
\begin{align*}
    y_i[k+1]= \sum_{j=1}^n [C_k]_{ij} y_j[k] + \nabla f_i (x_i[k+1]) - \nabla f_i(x_i[k]).
\end{align*}
Here, the mixing term $\sum_{j=1}^n [C_k]_{ij} y_j[k]$ helps the agents agree on the direction of $y$-variables, while the $\nabla f_i(x_i[k+1])-\nabla f_i(x_i[k])$ steer the $y$-variables towards the gradient direction. Therefore, we start by isolating the contraction in $\Syk$ coming from the mixing and the error introduced by the gradient update $\nabla f_i(x_i[k+1]) - \nabla f_i(x_i[k])$:
\begin{proposition}\label{proposition:grad-tracking-bound-with-consecutive-terms}
Let \Crefrange{assumption:objective-smoothness}{assumption:graph-connectivity} and Assumption~\ref{assumption:compatible_C}  hold. Then, we have for all $k\ge0$,
    \begin{align*}
        &\Sykp \\
        &\leq \tau_k \Syk +Lr_k\norm{\boldx[k+1]-\boldx[k]}_{\boldone},
    \end{align*}
    where $\boldone$ denotes the all ones vector.
\end{proposition}
Now, we have established that the agreement in the $y$-variables (i.e., $\Syk$) can be distorted by $\norm{\boldx[k+1]-\boldx[k]}_\boldone$. This is because as the $x$-variables change, the gradient evaluated at the previous location becomes less relevant. Hence, we now bound the error coming from this term:
\begin{proposition}\label{proposition:consecutive-terms-bound}
Let~\Crefrange{assumption:objective-strong-convexity}{assumption:compatible_C} hold. Let $\eta< \frac{1}{nL}$ and $\lambda \in (0, 1]$. Then, we have for all $k\ge0$,
    \begin{align*}
    \norm{\boldx[k+1]-\boldx[k]}_{\boldone} \leq \lambda \varphi_{k+1} (1+q_{k}(\eta, 1)) &\normphik{\boldx[k]-\boldx^*} \\
    + \left[\frac{1}{\sqrt{2}} \left( \varphi_k + \sigma_k \varphi_{k+1}\right)+ \eta \lambda L \varphi_k \varphi_{k+1} \sqrt{n}\right] &\Dxk\\
    +\eta \lambda \varphi_{k+1} &\Syk.
\end{align*}
\end{proposition}
Finally, we combine the results in \Cref{proposition:grad-tracking-bound-with-consecutive-terms} and \Cref{proposition:consecutive-terms-bound} to get the bound for $\Sykp$.
\begin{lemma}[Gradient Tracking Error Bound]
\label{lemma:gradient-tracking}
Let~\Crefrange{assumption:objective-strong-convexity}{assumption:compatible_C} hold. Let $\eta< \frac{1}{nL}$ and $\lambda \in (0, 1]$. Then, we have for all $k\ge0$,
    \begin{align*}
        &\Sykp\\
        &\leq \lambda Lr_k \varphi_{k+1} (1+q_{k}(\eta, 1)) \normphik{\boldx[k]-\boldx^*}\\
        &+Lr_k \left[ \left( \varphi_k + \sigma_k \varphi_{k+1}\right)+ \eta \lambda L \varphi_k \varphi_{k+1} \sqrt n \right] \Dxk \\
        &+ \left(\tau_k + \eta \lambda L r_k \varphi_{k+1}  \right)\Syk.
    \end{align*}
\end{lemma}
Similar to the consensus error, we get contraction in the gradient tracking error with coefficient $\tau_k$. All the other errors can be made small by choosing a smaller step size $\lambda$. This completes all the necessary results needed to establish \Cref{proposition:composite-relation}. Using \Cref{lemma:optimality-gap}, \Cref{lemma:consensus-error}, and \Cref{lemma:gradient-tracking}, and the upper bounds on $r_k,\varphi_k,$ etc. (see the paragraph preceding \Cref{proposition:composite-relation}), we obtain the composite relation matrix $M(\eta, \lambda).$

\subsection{Impossibility Results}
In this section, we give some theoretical results highlighting the need for including an extra step size $\lambda$. Consider the case where we set $\lambda=1$ in \Cref{eq:update_rule_c}, thus removing the lazy update, so that
$$z_i[k+1] = \proj{x_i[k+1]-\eta y_i[k+1]}.$$

We will establish that with this update rule, it is not possible to bound the term $\norm{\boldx[k+1]-\boldx[k]}$ such that
\begin{align*}
    \norm{\boldx[k+1]-\boldx[k]}&\leq c_1(\eta) \normphik{\boldx[k]-\boldx^*} \\&+ c_2(\eta) D(\boldx[k], \phi_k) + c_3(\eta) S(\boldy[k], \pi_k),
    \label{eq:consecutive-term-bound-structure}
\end{align*}
where $\lim_{\eta\to 0} c_1(\eta) = 0$ for every configuration of the problem. The term $\boldx[k+1]-\boldx[k]$ is essential for the analysis of the gradient tracking error since it is directly related to the term $\nabla f_i(x_i[k+1])-\nabla f_i(x_i[k])$ by both the L-smoothness and strong convexity. The following result shows that we cannot control the error contribution coming from this term by the optimality error by simply decreasing the step size. This problem persist even when the system consists of a single agent with perfect gradient knowledge, as in centralized projected gradient method.
\begin{lemma} 
    Assume that the function $f(x):\R^d \to \R$ is $\mu$-strongly convex and its gradient $\nabla f(x)$ is $L$-Lipschitz continuous. Moreover, assume that the constraint set $\mX\subseteq \R^d$ is convex and closed. Consider the sequence $\{x[k]\}_{k=1}^\infty$ generated by the centralized projected gradient method:
    \begin{align}
        x[k+1]= \proj{x[k+1]-\eta \nabla f(x[k])},
    \end{align}
    for some $x[0]\in \R^d$. Suppose that there exist a coefficient $c(\eta)$ that depends on $\eta$ such that for all $k\geq 0$ we have
\begin{align}
    \norm{x[k+1]-x[k]}&\leq c(\eta) \norm{x[k]-x^*},
\end{align}
where $x^*$ is the unique minimizer of $f(x)$ over  the set $\mX$. Then, there exists a function $f$ and a constraint set $\mX$ where 
$c(\eta) \geq b$ where $b>0$ for any $x[0]\in \mX \setminus \{x^*\}$ and any $c(\eta)$ with $\eta>0$.
    \label{lemma:pgd_impossibility}
\end{lemma}
\begin{proof}
    We prove this result by constructing an example. Let $f(x)=\frac{L}{2}x^2$ and $\mX = \{x \in \R \mid x \geq 1\}$. Then, the function $f$ is strongly convex with $L$-Lipschitz continuous gradients and the set $\mX$ is closed and convex. Notice that the update rule in this example is $x[k+1] = \proj{(1-\eta L) x[k]}$ since $\nabla f(x)=Lx$. Let $c(\eta)$ satisfy the following inequality:
    \begin{align*}
        \norm{x[k+1]-x[k]} \leq c(\eta) \norm{x-x^*}.
    \end{align*}
    Let the initial point $x[0]\in \mX \setminus \{x^*\}$. We split the proof into two cases. First, assume that $\eta\geq\frac{1}{L}$. Hence, we have $x[1] = \proj{(1-\eta L) x[0]}=1=x^*.$ Then,
    \begin{align*}
        \norm{x[1]-x[0]}&\leq c(\eta) \norm{x[0]-x^*}\\
        \norm{x^* - x[0]}&\leq c(\eta) \norm{x[0]-x^*}\\
        1&\leq c(\eta).
    \end{align*}
    Since $c(\eta)$ should satisfy the inequality \Cref{eq:impossiblilty_c_eta_ineq} for any $k\geq0$, it must be that $c(\eta)\geq 1$ for any $c(\eta)$. 
    
    Next, we consider the case where $0<\eta<\frac{1}{L}.$ Define the set $\mC_\eta=\{x \in \R \mid 1 < x \leq \frac{1}{1-\eta L}\}.$ For any $x[k]\in \mC_\eta$, we have $(1-\eta L)x[k]\leq 1$ since $x[k] \leq \frac{1}{1-\eta L}$. Therefore, $x[k+1]=x^*.$ Then, using similar steps to the proof with $\eta>\frac{1}{L}$ we have $c(\eta) \geq 1.$ The only remaining part is to show that for any $x[0]\in \mX \setminus \{x^*\}$, there exists an $x[k]\in \mC_\eta$. When $x[0]\in \mC_\eta,$ this is true trivially. Assume that $x[0]\notin \mC_\eta$, i.e., $x[k]>\frac{1}{1-\eta L}.$ First, notice that for any $x[k]\notin \mC_\eta$, $x[k+1]>1.$ We know that $x[k]$ should converge to $x^*=1$ since the iterates follow the projected gradient update rule with $\eta<\frac{1}{L}$ \cite[Chapter 7.2]{polyak1987introduction}. Then, there must be an $x[k]\in \mC_\eta$, which completes the proof.
\end{proof}
\begin{remark}
   Let~\Crefrange{assumption:objective-strong-convexity}{assumption:compatible_C} hold. Let there be a single agent in the system, i.e., $n=1.$ Then, if $\lambda=1$, the Projected Push-Pull algorithm in \Cref{eq:update_rule} is equivalent to the centralized projected gradient descent given in \Cref{lemma:pgd_impossibility}.
\end{remark}
Therefore, the impossibility results that we have shown for the projected gradient method also apply to the Projected Push-Pull algorithm.
\begin{corollary}
    \label{corollary:impossibility_consecutive_x}
    Let \Crefrange{assumption:objective-strong-convexity}{assumption:compatible_C} hold true. Assume that the agents follow the Projected Push-Pull algorithm given in \Cref{eq:update_rule} with $\lambda=1$. Suppose that there exist coefficients $c_1(\eta)$, $c_2(\eta)$, and $c_3(\eta)$ that depend on $\eta$ such that for all $k\geq 0$ we have
\begin{align}
    \label{eq:impossiblilty_c_eta_ineq}
    \norm{x[k+1]-x[k]} &\leq c_1(\eta) \normphik{\boldx[k] - \boldx^*}\\
    &+c_2(\eta)D(\boldx[k], \phi_k)+ c_3(\eta)S(\boldy[k], \pi_k).\nonumber
\end{align}
Then, there exist functions $f_i$, a constraint set $\mX$, and initial points $x_i[0]$ where $c_1(\eta) \geq b$ where $b>0$ for any $c(\eta)$ with $\eta>0$.
\end{corollary}
\begin{proof}
    Choose $f_i(x)=f_j(x)$ for all $i,j \in \mV.$ Consider a fully connected graph with $[R_{k}]_{ij}=[C_{k}]_{ij}=\frac{1}{n}$ for all $i,j \in \mV$ and for all $k\geq 0.$ Let each agent initialize the algorithm from the same point, i.e., $x_i[0]=x_j[0]$ for all $i,j \in \mV$. Then, the Projected Push-Pull algorithm is equivalent to following a centralized projected gradient descent for all agents. Moreover, we have $D(\boldx[k], \phi_k)=0$ and $S(\boldy[k], \pi_k)=0.$ Then, by \Cref{lemma:pgd_impossibility}, we know that there exist a function $f_i(x)$ and a constraint set $\mX$ such that $c_1(\eta)\geq b$ where $b>0$ for any $c_1(\eta)$ with $\eta>0.$
\end{proof}
Hence, we have established that we cannot fully control the bound on the term $\norm{\boldx[k+1]-\boldx[k]}$ by simply changing the step size $\eta$. This term has to arise in our analysis due to the definition of gradient tracking, which poses an important challenge to analyzing gradient tracking algorithms using projections. However, this problem does not happen when $\mX = \R^d$, i.e., when the problem is unconstrained. The main challenge in the constrained case is that the gradient at $x^*$ is typically non-zero, and therefore, we reach a pathological case where the agents do not slow down as they reach the optimal point. In the unconstrained case, as the agents reach the optimum, gradient also slows down since it vanishes. 

In a similar fashion to \Cref{corollary:impossibility_consecutive_x}, we have a fundamental limitation in the analysis of the consensus error when $\lambda=1$ in the algorithm. The following result shows this limitation.

\begin{lemma}
    \label{lemma:impossiblilty_consensus_bound}
    Let \Crefrange{assumption:objective-strong-convexity}{assumption:compatible_C} hold true. Assume that $n=2$, i.e. there are two agents in the system. Let the agents follow the Projected Push-Pull algorithm given in \Cref{eq:update_rule} with $\lambda=1$. Suppose that there exist coefficients $c_1(\eta)$, $c_2(\eta)$, and $c_3(\eta)$ that depend on $\eta$ such that for all $k\geq 0$ we have
\begin{align}
    D(\boldx[k+1], \phi_k) &\leq c_1(\eta) \normphik{\boldx[k] - \boldx^*}\\
    &+c_2(\eta)D(\boldx[k], \phi_k)+ c_3(\eta)S(\boldy[k], \pi_k).\nonumber
\end{align}
Then, there exist functions $f_1$, $f_2$, a constraint set $\mX$, and initial points $x_1[0]$ and $x_2[0]$ such that $\lim_{\eta \to 0} c_1(\eta) \neq 0$ for any $c_1(\eta)$.
\end{lemma}
\begin{proof}
    Let $f(x)= \frac{L}{2}x^2= \sum_{i=1}^2 f_i(x)$ where $f_1(x) = ~\pi_1 \frac{L}{2}x^2$ and $f_2(x) = \pi_2\frac{L}{2}x^2$. Let $\mX = \{x\in ~\R \mid x \geq 1\}$ be the closed and convex constraint set. Without loss of generality, assume $R_k=R$ and $C_k=C$ for all $k$. Furthermore, assume that $R$ is doubly stochastic and $C$ is a column stochastic matrix with the right eigenvector $\pi$ such that $C\pi=\pi$ with $\pi_1 > \pi_2$. Let $\eta \in \R$ with $0 < \eta < \frac{1}{L\pi_2}$ be arbitrary. Also, construct the initial conditions as $x_1[0]=x_2[0]=\frac{1}{1-\eta L \pi_1}$. Then, clearly, $D(\boldx[0], \phi)=S(\boldy[0], \pi)=0$. So, it must be that
    \begin{align*}
        D(\boldx[1], \phi)\leq c_1(\eta) \normphi{\boldx[0]-\boldx^*}.
    \end{align*}
    By construction, we have that $\normphi{\boldx[0]-\boldx^*}=\frac{\eta L \pi_1}{1-\eta L}$ and $D(\boldx[1], \phi)=\frac{1}{\sqrt{2}} \frac{\eta L}{1-\eta L} (\pi_1-\pi_2)$. Then, we must have
    \begin{align*}
        0<\frac{1}{\sqrt{2}}\eta L (\pi_1-\pi_2) &\leq c_1(\eta) \eta L \pi_1\\
        0 < \frac{1}{\sqrt{2}}(1- \frac{\pi_2}{\pi_1})&\leq c_1(\eta),
    \end{align*}
    which means that $\lim_{\eta \to 0} c_1(\eta) \neq 0$ as the relation above holds for any $\eta$ in the range $0 < \eta < \frac{1}{L\pi_2}$. 
\end{proof}
Essentially, \Cref{corollary:impossibility_consecutive_x} and \Cref{lemma:impossiblilty_consensus_bound} show that no matter how we bound the error terms, we cannot gain full control over the non-diagonal entries of the composite relation matrix $M(\eta, \lambda)$,
with $\lambda=1$. These entries are essential in controlling the spectral radius of $M(\eta, 1)$, and guaranteeing convergence. However, being able to choose a $\lambda$ in range $(0,1]$ gives us more flexibility.
\section{NUMERICAL STUDIES}
 In this section, we empirically demonstrate the convergence of our algorithm on a sample optimization problem and also investigate the effect of the graph properties and mixing times of matrices $R_k$ and $C_k$ on convergence.
\subsection{Convergence of Protocol on Time Varying Communication Networks}
\textbf{Optimization Problem:} We have $n=50$ agents. Agent $i$ has objective function $(x_i - x_i^c)^{\intercal} P_i (x_i - x_i^{c})$  where $x_i^c \in \R^d$ is the center of the quadratic and the matrix $P_i \in \R^{d \times d}$ is positive definite with $\mu I \preceq P_i \preceq L I$ for some $\mu$ and $L$. Hence, the global objective is $\sum_{i=1}^n (x_i-x_i^c)^\intercal P_i (x_i - x_i^c)$.

For this experiment, we set $d=2$ and sample the coordinates of $x_i^c$ from the uniform distribution $\mU[-2,8]$ independently for each $i$. Similarly, we set $P_i$ to be diagonal for each $i$, and sample each entry in the diagonal independently from $\mU[0,1]$. Notice that the objectives are strongly convex and smooth. We set $\mX = \overline{B((6,6), 2)}$, the closed disk around $(6,6)$ with radius $2$. This setup is likely to result in a optimal point $x^*$ in the boundary of $\mX$, which helps us demonstrate the effectiveness of our algorithm when $\nabla f(x^*) \neq 0$, in contrast with the unconstrained setting.

\textbf{Communication graphs:} We construct the time-varying graphs by iterating over the sequence of graphs $\{\mG_1, \dots, \mG_T\}$ where $T=5$.  That is, $\mG_k = \mG_{(k-1 \mod T) + 1}$. We generate $\mG_i$ independently for each $i\in [T]$ as follows. For all $j,l \in \mV$ we add the edge $(j,l)$ to $\mE_i$ with probability $p=0.1$. We regenerate the graph if it is not strongly connected.

\textbf{Mixing matrices $R_k$ and $C_k$:} Let $\mN_i^{\sf in}[k]$ and $\mN_{i}^{\sf out}[k]$ denote the in/out neighborhoods of agent $i$ at time $k$ respectively. That is, $j\in \mN_i^{\sf in}[k]$ if and only if $(j,i) \in \mE_k$ and $j \in \mN_i^{\sf out}[k]$ if and only if $(i,j) \in \mE_k$. Agent $i$ sets the $i$'th row of $R_k$ and the $i$'th column of $C_k$ as follows:
\begin{align*}
    [R_k]_{ij} &= \begin{cases}
        \frac{1}{|\mN_i^{\sf in}[k]|+1} &\text{ if } j \in \mN_i^{\sf in}[k] \text{ or } j = i\\
        0&\text{ otherwise }
    \end{cases},\\
    [C_k]_{ji}& \begin{cases}
        \frac{1}{|\mN_i^{\sf out}[k]|+1} &\text{ if } j \in \mN_i^{\sf out}[k] \text{ or } j = i\\
        0&\text{ otherwise }
    \end{cases}.
\end{align*}

\textbf{Optimization parameters}: We set $\eta = 0.5,  \lambda = 0.7$. We initialize $x_i[0]$ by sampling each coordinate from $\mU[0,10]$ and projecting onto $\mX$. 

We plot the error terms\footnote{One minor difference between these error terms and the ones used in the analysis is that we cannot compute the sequence $\{\phi_k\}$ as used in the analysis. Therefore, we choose $\phi_k = \frac{1}{n}\boldone$ for all $k$. We initialize $\pi_0 = \frac{1}{n} \boldone$ and let $\pi_{k+1} = C_k \pi_k$.} in \Cref{fig:demonstration-experiment}.  As we can see in \Cref{fig:demonstration-experiment}, all of the error goes to $0$ with geometric (linear) rate. The convergence rates (slopes) of all the terms are similar, which highlights the interdependence of these errors.
\begin{figure}
    \centering
    \includegraphics[width=0.45\textwidth]{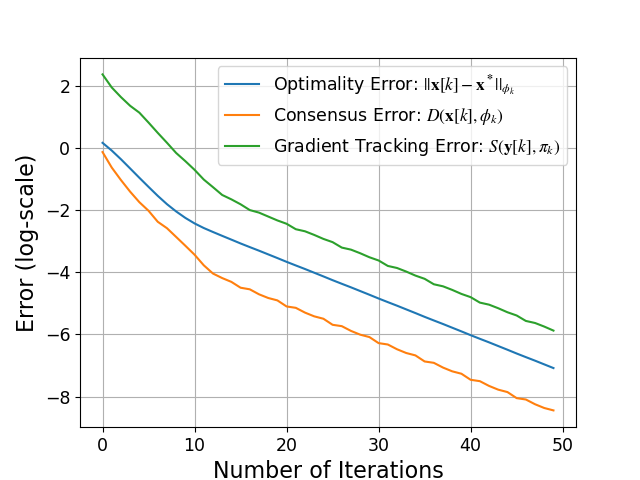}
    \caption{Optimality, consensus, and gradient tracking errors vs. number of iterations on a log-scale. The errors converge to $0$ as geometric (linear) rate with similar convergence rates.}
    \vspace{-5pt}
    \label{fig:demonstration-experiment}
\end{figure} 

\subsection{Effect of Graph Type on Convergence}

Because the contraction of matrices $R_k$ and $C_k$ are related to the communication graph as given in \Cref{lemma:R-contraction} and \Cref{lemma:C-contraction}, graph type will affect the convergence rate of our algorithm. Therefore, we investigate the effect of graph type on the convergence rate. In this section, we set $n=15$ and $T=1$ (static graphs) but otherwise use the same setup for the objective function and constraint set as the previous section. We generate three different graph types as follows:
1) \textit{Random}: Same as described in the previous section.
2) \textit{Cyclic}:~We have $\mV=[n]$ and $(i,i+1) \in \mE$ for $i=1, \dots, n-1$ and $(n, 1) \in \mE$. 
3) \textit{Unbalanced}:~Graph where certain nodes have very high in-degrees and low out-degrees and vice versa.

We set $R_k$ and $C_k$ to be compatible for each graph as described in the previous section. We fix $\eta = 0.5$ for all graphs, and for each graph, we choose $\lambda$ to be the largest value that allows convergence. Specifically, the random graph requires $\lambda = 0.15$, the unbalanced graph requires $\lambda = 0.3$, and the cyclic graph requires $\lambda=0.6$ to have convergence. The comparison of the convergence rate between these graphs is shown in \Cref{fig:graph-comparison}. 
\begin{figure}
    \centering
    \includegraphics[width=0.45\textwidth]{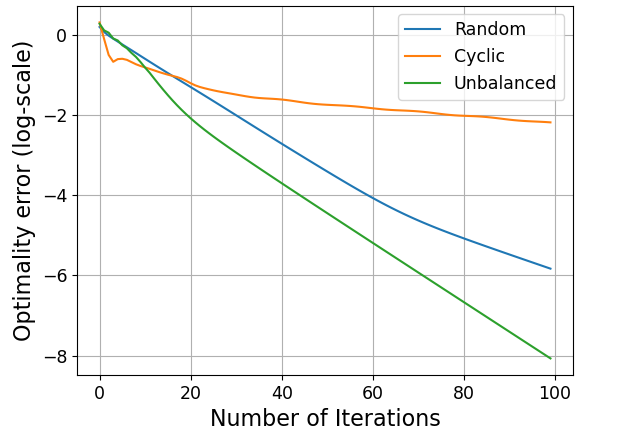}
    \caption{Optimality error $\normphik{\boldx[k] - \boldx^*}$ vs. number of iterations $k$ on a log-scale for different graph types. Even though the error converges to 0 geometrically for all graphs, unbalanced and random graphs have a much higher convergence rate compared to cyclic graphs, which are slowly mixing.}
    \label{fig:graph-comparison}
    \vspace{-20pt}
\end{figure}
We see that the random and unbalanced graphs have a faster convergence due to having higher-connectivity. 

\section{CONCLUSIONS AND FUTURE WORK}
In this paper, we introduce the Projected Push-Pull algorithm, which combines gradient tracking and projected gradient method to address distributed constrained optimization problems on time-varying directed graphs. We prove that our algorithm achieves geometric convergence with sufficiently small step sizes. We derive explicit bounds for the step sizes based on the characteristics of the cost functions and the communication graph. Moreover, we provide additional theoretical results showing that having a non-zero gradient at the optimal point in constrained problems poses additional challenges in the analysis of gradient tracking methods employing projection. Finally, we demonstrate the geometric convergence of our algorithm via numerical studies over various graph types. An interesting direction for future work is the incorporation of random and adversarial noise.


\bibliographystyle{IEEEtran} 
\bibliography{references}

\newpage
\onecolumn

\appendix
\label{sec:appendix}
\subsection{Gradient Contraction Results}
\begin{lemma}
    Let $f: \R^d \to \R$ be $\mu$-strongly convex and $L$-smooth. Then, for any $x, y \in \R^d$ we have
    \begin{align*}
        \langle \nabla f(x) -\nabla f(y), x-y \rangle \geq \frac{\mu L}{\mu + L} \normsq{x-y} + \frac{1}{\mu + L} \normsq{\nabla f(x) - \nabla f(y)}
    \end{align*}
    \label{lemma:grad-inner-product}
\end{lemma}
\begin{proof}
    See Lemma 3.11 in \cite{bubeck2015convex}
\end{proof}

\subsubsection{Proof of \cref{lemma:projected-gd-contraction}}
    First, we plug in the definition of $\mT_\eta$ and use non-expansivity, then expand the squared norm
    \begin{align*}
        \normsq{\mT_\eta(x)-\mT_\eta(y)}&=\normsq{\proj{x-\eta \nabla f(x)}- \proj{y-\eta \nabla f(y)}}\\
        &\leq \normsq{x-y - \eta (\nabla f(x) - \nabla f(y))}\\
        &= \normsq{x-y} - 2\eta \langle x-y, \nabla f(x) - \nabla f(y)\rangle +\eta ^2\normsq{\nabla f(x) -\nabla f(y)}.
    \end{align*}
Now, choose $\eta < \frac{2}{\mu + L}$. Then, let $L' = \frac{2}{\eta} - \mu \geq L$ and $\mu'= \mu$. Then, it follows that $f$ is $\mu$' strongly convex and $L'$ smooth. Applying \cref{lemma:grad-inner-product} to the inner product, we get
\begin{align*}
    \normsq{\mT_\eta(x)- \mT_\eta (y)}&\leq \left(1 -2 \eta \frac{\mu' L'}{\mu' + L'}\right) \normsq{x-y} + \left(\eta ^2 - 2\eta \frac{1}{\mu'+L'}\right) \normsq{\nabla f(x) - \nabla f(y)}
\end{align*}
Then, notice that $\frac{1}{\mu' + L'}=\frac{\eta }{2}$. Plugging in, we get
\begin{align*}
    \normsq{\mT_\eta(x)- \mT_\eta (y)}&\leq \left(1 -\eta^2 \mu' L'\right)\normsq{x-y}+\left(\eta^2 - 2 \eta \frac{\eta}{2}\right) \normsq{\nabla f(x)- \nabla f(y)} \\
    &= (1-\eta^2 \mu (2/\eta - \mu))\normsq{x-y}\\
    &=(1-2\eta \mu +\eta^2 \mu^2 ) \normsq{x-y}\\
    &=(1-\eta \mu)^2 \normsq{x-y}.
\end{align*}
Taking square roots of both sides gives us the desired result. Notice that $|1-\eta \mu| = 1 -\eta \mu$ when $\eta < \frac{2}{\mu + L}< \frac{1}{\mu}$.

The following corollary follows directly from this lemma.
\begin{corollary}[Contraction with lazy updates]
Let $\mT_{\eta, \lambda}(x) = (1-\lambda)x + \lambda \proj{x-\eta \nabla f(x)}$ and $\eta < \frac{2}{\mu + L}$ and $\lambda \in (0, 1]$. Then,
\begin{align*}
    \norm{\mT_{\eta, \lambda}(x)-\mT_{\eta, \lambda}(y)} &\leq q(\eta, \lambda) \norm{x-y}
\end{align*}
where $q(\eta, \lambda)= 1-\lambda \eta \mu$ for $\eta < \frac{2}{\mu + L}$ and any $\lambda \in (0, 1]$.
\label{corollary:gradient-contraction-with-lazy}
\end{corollary}

\begin{proof}
    Immediately follows from \cref{lemma:projected-gd-contraction} after writing the definition of $\mT_{\eta, \lambda}$ and applying triangle inequality.
\end{proof}
\begin{remark}
    Under the conditions in \cref{lemma:projected-gd-contraction}, we have that $\norm{\mT_\eta(x)-x^*} \leq q(\eta) \norm{x-x^*}$ where $x^*$ is the unique minimizer of $f$ over the set $\mX$. This is because $x^* = T_\eta (x^*)$ for all $\eta > 0$.
\end{remark}

\subsection{Bounding the Optimality Gap}

\subsubsection{Proof of \cref{proposition:imperfect_gradient_error}}
    First, let us analyze each term in the norm $\normphik{\boldz[k]-\boldw[k]}$
    \begin{align*}
        \normsq{z_i[k]-w_i[k]}&=\normsq{(1-\lambda)x_i[k]+\lambda \proj{x_i[k]-\eta y_i[k]}-(1-\lambda)x_i[k]-\lambda \proj{x_i[k]-\eta n\piki \nabla f(x_i[k])}}\\
        &=\lambda^2 \normsq{\proj{x_i[k]-\eta y_i[k]}- \proj{x_i[k]-\eta n\piki \nabla f(x_i[k])}}\\
        &\leq \lambda^2 \normsq{x_i[k]-\eta y_i[k]-x_i[k]-\eta n\piki \nabla f(x_i[k])}\\
        &=\lambda^2 \eta^2 \normsq{y_i[k]-n\piki \nabla f(x_i[k])}.
    \end{align*}
    Let $g_i[k]$ be the sum of all gradients evaluated at $x_i[k]$, that is $g_i[k] \triangleq n\nabla f(x_i[k])=\sum_{l=1}^n \nabla f_l(x_i[k])$. Similarly, define $h_i[k]$ be the sum of agents' $y$-values. That is, $h_i[k] \triangleq \sum_{l=1}^n y_l[k] = \sum_{l=1}^n \nabla f_l(x_l[k])$. Then, using the above inequality, we have 
    \begin{align*}
        \normphik{\boldz[k]-\boldw[k]}&\leq \lambda \eta \normphik{\boldy[k]-\diag(\pi_k)\boldg[k]}\\
        &\leq \lambda \eta \normphik{\boldy[k]-\diag(\pi_k)\boldh[k]} + \lambda \eta \normphik{\diag(\pi_k)\boldh[k]-\diag(\pi_k)\boldg[k]}.
    \end{align*}
     Here, the first term is related to the consensus in $y$-values, and the second term is related to the consensus between $x$-values. Hence, we analyze the first term as follows
    \begin{align*}
        \normphik{\boldy[k]-\diag(\pi_k)\boldh[k]}&=\sqrt{\sum_{i=1}^n\phiki \normsq{y_i[k]-\piki \sum_{l=1}^n y_l[k]}}\\
        &=\sqrt{\sum_{i=1}^n \phiki \piki^2 \normsq{\frac{y_i[k]}{\piki}-\sum_{l=1}^n y_l[k]}}\\
        &\leq \sqrt{\max_i(\phiki \piki)} \sqrt{\sum_{i=1}^n \piki \normsq{\frac{y_i[k]}{\piki}-\sum_{l=1}^n y_l[k]}}\\
        &=\sqrt{\max_i (\phiki \piki)} \Syk \leq \Syk,
    \end{align*}
    since $\phiki \piki \leq 1$ for all $i, k$. Now, we focus our attention to the second term, $\normphik{\diag(\pi_k)\boldh[k]-\diag(\pi_k)\boldg[k]}$. 
    \begin{align*}
        \normphik{\diag(\pi_k)\boldh[k]-\diag(\pi_k)\boldg[k]}&=\sqrt{\sum_{i=1}^n \phiki \normsq{\piki\sum_{l=1}^n y_l[k]-\piki\sum_{l=1}^n \nabla f_l(x_i[k])}}\\
        &=\sqrt{\sum_{i=1}^n \phiki \piki^2 \normsq{\sum_{l=1}^n \nabla f_l(x_l[k])-\sum_{l=1}^n \nabla f_l(x_i[k])}}\\
        &\overset{(a)}{\leq} \sqrt{\max_i \piki^2} \sqrt{n\sum_{i=1}^n \phiki \sum_{l=1}^n \normsq{\nabla f_l(x_l[k])-\nabla f_l(x_i[k])}}\\
        &\overset{(b)}{\leq} \sqrt{\frac{n}{\min \phi_k}}\sqrt{\sum_{i=1}^n \sum_{l=1}^n \phiki \phikl \normsq{\nabla f_l(x_l[k])-\nabla f_l(x_i[k])}}\\
        &\leq L\sqrt{\frac{n}{\min \phi_k}} \sqrt{\sum_{i=1}^n \sum_{l=1}^n \phiki\phikl\normsq{x_l[k]-x_i[k]}}\\
        &=L\sqrt{\frac{n}{\min \phi_k}}\Dxk.
    \end{align*}
    Here, we used  $\normsq{\sum_{i=1}^n x_i}\leq n \sum_{i=1}^n \normsq{x_i}$ in (a) and $\sqrt{\max_i \piki^2}\leq 1$ in (b).
    Plugging back to the initial expression, we obtain
    \begin{align*}
        \normphik{\boldz[k]-\boldw[k]}\leq \lambda \eta L \sqrt{\frac{n}{\min \phi_k}} \Dxk+\lambda  \eta \Syk,
    \end{align*}
    as desired.

\subsubsection{Proof of \cref{lemma:optimality-gap}}
    Notice that we have $\normphikp{\boldx[k+1]-\boldx^*} \leq \normphik{\boldz[k]-\boldx^*}$ from \cref{lemma:R-contraction} with $u=x^*$. Then, we write
    \begin{align*}
        \normphik{\boldz[k]-\boldx^*}\leq \normphik{\boldw[k]-\boldx^*} + \normphik{\boldw[k]-\boldz[k]}.
    \end{align*}
    Now, the first term represents the contractions coming from projected gradient under global gradient knowledge. Since $\eta < \frac{1}{L}$ and $\min \pi_k n \leq 1$, we have $\eta n \min \pi_k < \frac{1}{L}\leq \frac{2}{\mu + L}$. Using \cref{corollary:gradient-contraction-with-lazy}, we have
    \begin{align*}
    \normphik{\boldw[k]-\boldx^*} &= \sqrt{\sum_{i=1}^n \phiki \normsq{w_i[k] -x^*}}\\
    &\leq \sqrt{\sum_{i=1}^n \phiki q(\eta n\piki, \lambda)^2 \normsq{x_i[k]-x^*}}\\
    &\leq q_k(\eta, \lambda) \normphik{\boldx[k]-\boldx^*},
\end{align*}
which gives
\begin{align*}
    \normphikp{\boldx[k+1]-\boldx^*} \leq q_k(\eta, \lambda) \normphik{\boldx[k]-\boldx^*} + \normphik{\boldz[k]-\boldw[k]}.
\end{align*}
Combining with the bound for $\normphik{\boldz[k]-\boldw[k]}$ from \cref{proposition:imperfect_gradient_error}, we get the desired result.

\subsection{Bounding Consensus Error}

\subsubsection{Proof of \cref{proposition:removing-gradient-error-consensus}}

   We want to relate $z_i[k]$ to $w_i[k]$, where $w_i[k]=(1-\lambda)x_i[k]-\lambda\proj{x_i-\eta n\piki \nabla f(x_i[k])}$ as defined in \cref{sec:bounding_optimality_gap}. Notice that
   \begin{align}
       \normsq{z_i[k]-z_j[k]}&=\normsq{z_i[k]-w_i[k]+w_i[k]-w_j[k]+w_j[k]-z_j[k]}
       \label{eq:quantifying-consensus-error-w}
   \end{align}
   where the terms $w_i[k]-z_i[k]$ and $w_j[k]-z_j[k]$ quantify the gradient errors and $w_i[k]-w_j[k]$ gives us the consensus error under perfect gradient knowledge. Writing the definitions of $D(\boldz[k], \phi_k)$ and $D(\boldw[k], \phi_k)$ and using the triangle inequality for the $\ell_2$ norm, we have
   \begin{align*}
       D(\boldz[k], \phi_k)&=\sqrt{\sum_{i=1}^n \sum_{j=1}^n \phiki \phikj\normsq{z_i[k]-z_j[k]}}\\
       &=\sqrt{\sum_{i=1}^n \sum_{j=1}^n \phiki \phikj\normsq{z_i[k]-w_i[k]+w_i[k]-w_j[k]+w_j[k]-z_j[k]}}\\
       &\overset{(a)}{\leq} 2\sqrt{\sum_{i=1}^n \sum_{j=1}^n \phiki \phikj\normsq{z_i[k]-w_i[k]}} + \sqrt{\sum_{i=1}^n \sum_{j=1}^n \phiki \phikj\normsq{w_i[k]-w_j[k]}}\\
       &\overset{(b)}{=} 2 \normphik{\boldz[k]-\boldw[k]} + D(\boldw[k], \phi_k),
   \end{align*}
    where in (a) we used the symmetry between $z_i[k]-w_i[k]$ and $z_j[k]-w_j[k]$, and in (b) we used the fact that $\sum_{j=1}^n \phikj = 1$.

\subsubsection{Proof of \cref{proposition:w-consensus-bound}}
    We will need to expand the definition of $D(\boldw[k], \phi_k)$ first. The main challenge of this part of the analysis is that $w_i[k]$ and $w_j[k]$ represent two agents taking gradient steps with different step sizes. Hence, to capture this, we define the following term: Let $r_{ij}[k]=(1-\lambda)x_j[k]-\lambda \proj{x_j[k]-\eta n \piki\nabla f(x_j[k])}$, i.e., applying update rule with global gradient to $x_j[k]$ with stepsize $\eta n [\pi_k]_i$.
    \begin{align}
        D(\boldw[k], \phi_k)&=\sqrt{\sum_{i=1}^n \sum_{j=1}^n \phiki \phikj\normsq{w_i[k]-w_j[k]}}\nonumber\\
        &=\sqrt{\sum_{i=1}^n \sum_{j=1}^n \phiki \phikj \normsq{w_i[k]-r_{ij}[k]+r_{ij}[k]-w_j[k] }}\nonumber\\
        &\leq \sqrt{\sum_{i=1}^n \sum_{j=1}^n \phiki \phikj \normsq{w_i[k]-r_{ij}[k]}} +\sqrt{\sum_{i=1}^n \sum_{j=1}^n \phiki\phikj \normsq{r_{ij}[k]-w_j[k] }}, \label{eq:consensus-w-split}
    \end{align}
    where the last line follows from the triangle inequality in the $\ell_2$ norm. Now, we have two terms. One is the consensus error between two agents that take a gradient step on the global gradient with the same step size. The other is the error in a single agent induced by taking different step sizes from the same point. We initially turn our attention to the first term. We have
    \begin{align*}
        &\sqrt{\sum_{i=1}^n \sum_{j=1}^n \phiki \phikj \normsq{w_i[k]-r_{ij}[k]}} \\ &=\sqrt{\sum_{i=1}^n \sum_{j=1}^n \phiki \phikj \normsq{(1-\lambda)(x_i[k]-x_j[k])+ \lambda(\proj{x_i[k]-\eta n[\pi_k]_i \nabla f(x_i[k])}-\proj{x_j[k]-\eta n[\pi_k]_i \nabla f(x_j[k]))}}}\\
        &\overset{(a)}{\leq} (1-\lambda + \lambda \max_{i} q(\eta n [\pi_k]_i)) \sqrt{\sum_{i=1}^n\sum_{j=1}^n \phiki\phikj \normsq{x_i[k]-x_j[k]}}= q_k(\eta, \lambda)\Dxk,
    \end{align*}
    where (a) follows from the contraction of projected gradient descent with lazy updates, i.e., \cref{corollary:gradient-contraction-with-lazy}. Now, we proceed to analyze the second term in \cref{eq:consensus-w-split}. Notice that
    \begin{align*}
        \normsq{r_{ij}[k]-w_j[k]}&=\normsq{(1-\lambda)x_j[k]+\lambda \proj{x_j[k]-\eta n[\pi_k]_i\nabla f(x_j[k])}-(1-\lambda)x_j[k]-\lambda \proj{x_j[k]-\eta n\pikj \nabla f(x_j[k])}}\\
        &=\lambda^2 \normsq{\proj{x_j[k]-\eta n[\pi_k]_i \nabla f(x_j[k])}-\proj{x_j[k]-\eta n\pikj \nabla f(x_j[k])}}.
    \end{align*}
    Hence, we have
    \begin{align*}
        &\sqrt{\sum_{i=1}^n \sum_{j=1}^n \phiki \phikj \normsq{r_{ij}[k]- w_j[k]}} \\ & = \lambda \sqrt{\sum_{i=1}^n \sum_{j=1}^n \phiki \phikj \normsq{\proj{x_j[k]-\eta n[\pi_k]_i \nabla f(x_j[k])}-\proj{x_j[k]-\eta n\pikj \nabla f(x_j[k])}}} \\
        &\leq \lambda \sqrt{\sum_{i=1}^n \sum_{j=1}^n \phiki \phikj \normsq{\proj{x_j[k]-\eta n[\pi_k]_i \nabla f(x_j[k])}-x^* + x^*-\proj{x_j[k]-\eta n\pikj \nabla f(x_j[k])}}} \\
     &\leq \lambda \sqrt{\sum_{i=1}^n \sum_{j=1}^n \phiki \phikj \normsq{\proj{x_j[k]-\eta n[\pi_k]_i \nabla f(x_j[k])}-x^*}}\\
     &\hspace{50mm}+ \lambda \sqrt{\sum_{i=1}^n \sum_{j=1}^n \phiki \phikj \normsq{x^*-\proj{x_j[k]-\eta n\pikj \nabla f(x_j[k])}}}.
    \end{align*}
    From the projected gradient contraction (\cref{lemma:projected-gd-contraction}), we have
    \begin{align*}
        \lambda \sqrt{\sum_{i=1}^n \sum_{j=1}^n \phiki \phikj \normsq{\proj{x_j[k]-\eta n\pikj \nabla f(x_j[k])}-x^*}} &\leq \lambda  \max_i q(\eta n[\pi_k]_i ) \normphik{\boldx[k]-\boldx^*}.
    \end{align*}
    Doing the same with the second term, we get
    \begin{align*}
        \sqrt{\sum_{i=1}^n \sum_{j=1}^n \phiki \phikj \normsq{r_{ij}[k]-w_j[k]}}&\leq 2\lambda q_k(\eta, 1) \normphik{\boldx[k]-\boldx^*}.
    \end{align*}
    Combining this expression with the earlier bound we get
    \begin{align*}
        D(\boldw[k], \phi_k)\leq 2 \lambda q_k(\eta, 1) \normphik{\boldx[k]-\boldx^*} + q_{k}(\eta, \lambda) \Dxk,
    \end{align*}
    as desired.

\subsubsection{Proof of \cref{lemma:consensus-error}}
First, notice that $\Dxkp \leq\sigma_k  D(\boldz[k], \phi_k)$ from \cref{lemma:R-contraction}. Then, from \cref{proposition:removing-gradient-error-consensus}, we have
    \begin{align*}
        \Dxkp \leq \sigma_k D(\boldz[k], \phi_k)\leq \sigma_k D(\boldw[k], \phi_k) + 2\sigma_k \normphik{\boldz[k] - \boldw[k]}.
    \end{align*}
    Then, plugging in the bound for $D(\boldw[k], \phi_k)$ as found in \cref{proposition:w-consensus-bound}, we have
    \begin{align*}
        \Dxkp &\leq  2\sigma_k \lambda q_k(\eta, 1) \normphik{\boldx[k]-\boldx^*} + \sigma_k q_{k}(\eta, \lambda) D(\boldx[k], \phi_k) + 2\sigma_k \normphik{\boldz[k]-\boldw[k]}.
    \end{align*}
    Finally, plugging in the bound for $\normphik{\boldz[k] - \boldw[k]}$ in \cref{proposition:imperfect_gradient_error}, we get
    \begin{align*}
        \Dxkp \leq 2 \lambda\sigma_k q_k(\eta, 1) \normphik{\boldx[k]-\boldx^*} + (\sigma_kq_{k}(\eta, \lambda) + 2\eta \lambda \sigma_k L  \varphi_k\sqrt{n}) \Dxk  + 2\eta \lambda \sigma_k \Syk.
    \end{align*}
    as desired.

\subsection{Gradient Tracking Error Bound}
\subsubsection{Proof of \cref{proposition:grad-tracking-bound-with-consecutive-terms}}

Let $v_i[k]=\sum_{j=1}^n [C_k]_{ij}y_j[k]$ and stack these vectors in $\boldv[k]$. Similarly, stack the gradients in $\nabla F(\boldx[k])=(\nabla f_1(x_1[k]), \dots, \nabla f_n(x_n[k]))^T$. Notice that $S(\boldy[k+1], \pi_{k+1}) = \normpikp{\diag(\pi_{k+1})^{-1} \boldy[k+1]-\boldh[k+1]}$ by definition. Then writing the expression for $\boldy[k+1]$ from the update rule, we get
    \begin{align*}
        &\diag(\pi_{k+1})^{-1} \boldy[k+1]-\boldh[k+1]\\
        &=\diag(\pi_{k+1})^{-1}\boldv[k]+\diag(\pi_{k+1})^{-1}\nabla F(\boldx[k+1]) - \diag(\pi_{k+1})^{-1}\nabla F(\boldx[k]) - \boldh[k+1]\\
        &\overset{(a)}{=}\diag(\pi_{k+1})^{-1}\boldv[k]-\boldh[k]+\diag(\pi_{k+1})^{-1}\nabla F(\boldx[k+1]) - \diag(\pi_{k+1})^{-1}\nabla F(\boldx[k]) + \boldh[k] - \boldh[k+1],
    \end{align*}
    where in (a) we added and subtraced $\boldh[k]$. Now, taking the $\normpikp{\cdot}$ norm of both sides and applying triangle inequality, we get
    \begin{align*}
        &\normpikp{\diag(\pi_{k+1})^{-1}\boldy[k+1]-\boldh[k+1]}\\
        &\leq \normpikp{\diag(\pi_{k+1})^{-1}\boldv[k]-\boldh[k]}  + \normpikp{\diag(\pi_{k+1})^{-1}(\nabla F(\boldx[k+1])-\nabla F(\boldx[k]))}
        + \normpikp{\boldh[k+1]-\boldh[k]} .
    \end{align*}
    Now, we know the contraction of the first term from the consensus due to the matrix $\boldC$ (\cref{lemma:C-contraction}), which gives us
    \begin{align*}
        \normpikp{\diag(\pi_{k+1})^{-1}\boldv[k]-\boldh[k]} \leq \tau_k \normpik{\diag(\pi_{k})^{-1}\boldy[k]-\boldh[k]}=\tau_k\Syk.
    \end{align*}
    Hence, we will now analyze the other two terms. Notice that
    \begin{align*}
        \normpikp{\diag(\pi_{k+1})^{-1} (\nabla F(\boldx[k+1])-\nabla F(\boldx[k]))}&=\sqrt{\sum_{i=1}^n \pikpi \normsq{\frac{\nabla f_i(x_i[k+1]) - \nabla f_i (x_i[k])}{\pikpi}}}\\
        &\leq  \sqrt{\frac{1}{\min \pi_{k+1}}} \sqrt{\sum_{i=1}^nL^2\norm{x_i[k+1]-x_i[k]}^2}\\
        &\leq L \sqrt{\frac{1}{\min \pi_{k+1}}} \norm{\boldx[k+1]-\boldx[k]}_{\boldone}.
    \end{align*}
    Similarly, we have
    \begin{align*}
        \normpikp{\boldh[k+1]-\boldh[k]}&=\sqrt{\sum_{i=1}^n \pikpi \normsq{\sum_{l=1}^n\nabla f_l(x_l[k+1]) - \nabla f_l(x_l[k])}}\\
        &\overset{(a)}{=}\norm{\sum_{l=1}^n \nabla f_l(x_l[k+1])-\nabla f_l(x_l[k])}\\
        &\leq  \sum_{l=1}^n \norm{\nabla f_l(x_l[k+1]) - \nabla f_l(x_l[k])}\\
        &\leq L \sum_{l=1}^n \norm{x_l[k+1]-x_l[k]}\\
        &\leq L\sqrt{n} \norm{\boldx[k+1]-\boldx[k]}_{\boldone},
    \end{align*}
    where (a) follows from the stochasticity of $\pi_{k+1}$. Letting $r_k=\sqrt{\frac{1}{\min \pi_{k+1}}}+\sqrt{n}$ and combining everything we have found, we obtain
    \begin{align*}
        \Sykp\leq \tau_k \Syk +Lr_k\norm{\boldx[k+1]-\boldx[k]}_{\boldone},
    \end{align*}
    as desired.

 Hence, we bound the distance between consecutive terms as follows.

\subsubsection{Proof of \cref{proposition:consecutive-terms-bound}}

We first write the definition of $\norm{\boldx[k+1]-\boldx[k]}_{\boldone}$ and expand $x_i[k+1]$ to get
    \begin{align}
        \norm{\boldx[k+1]-\boldx[k]}_{\boldone}&=\sqrt{\sum_{i=1}^n \normsq{x_i[k+1]-x_i[k]}}\nonumber\\
        &=\sqrt{\sum_{i=1}^n \normsq{\sum_{j=1}^n [R_k]_{ij}z_j[k]-x_i[k]}}\nonumber\\
        &=\sqrt{\sum_{i=1}^n \normsq{\sum_{j=1}^n [R_k]_{ij}z_j[k]-\sum_{j=1}^n [R_k]_{ij}x_j[k]+\sum_{j=1}^n [R_k]_{ij}x_j[k]-x_i[k]}}\nonumber\\
        &\overset{(a)}{\leq} \sqrt{\sum_{i=1}^n  \normsq{\sum_{j=1}^n [R_k]_{ij} (z_j[k]-x_j[k])}} + \sqrt{\sum_{i=1}^n  \normsq{\sum_{j=1}^n [R_k]_{ij} x_j[k]-x_i[k]}}.
         \label{eq:consecutive-terms-split}
    \end{align}
    Where we used the triangle equality for the $\ell_2$ norm in (a). Now, we will analyze the two terms that arise in the last line separately. Let's first look at the second term. Let $\hat x[k] = \sum_{i=1}^n \phiki x_i[k]$ and $\hat \boldx[k] = (\hat x[k], \dots, \hat x[k])$. Then,

    \begin{align}
        \sqrt{\sum_{i=1}^n \normsq{\sum_{j=1}^n [R_k]_{ij}x_j[k] - x_i[k]}}&\leq \sqrt{\sum_{i=1}^n \normsq{\sum_{j=1}^n [R_k]_{ij}x_j[k] - \hat x[k]}} + \sqrt{\sum_{i=1}^n \normsq{x_i[k]-\hat x[k] }}\nonumber\\
        &\leq \sqrt{\frac{1}{\min \phi_{k+1}}}\sqrt{\sum_{i=1}^n \phikpi\normsq{\sum_{j=1}^n[R_k]_{ij}x_j[k] - \hat x[k]}} + \sqrt{\frac{1}{\min \phi_k}} \sqrt{\sum_{i=1}^n \phiki\normsq{x_i[k]-\hat x[k]}}\nonumber\\
        &\overset{(a)}{\leq} \sqrt{\frac{1}{\min \phi_{k+1}}} \sigma_k \sqrt{\sum_{i=1}^n \phiki \normsq{x_i[k]-\hat x[k]}}+ \sqrt{\frac{1}{\min \phi_k}} \sqrt{\sum_{i=1}^n \phiki \normsq{x_i[k] -\hat x[k]}}\nonumber\\
        &=\left(\sqrt{\frac{1}{\min \phi_k}} + \sigma_k \sqrt{\frac{1}{\min \phi_{k+1}}} \right)\sqrt{\sum_{i=1}^n \phiki \normsq{x_i[k] -\hat x[k]}} \nonumber\\
        &\overset{(b)}{=} \frac{1}{\sqrt{2}} \left(\sqrt{\frac{1}{\min \phi_k}} + \sigma_k \sqrt{\frac{1}{\min \phi_{k+1}}} \right)\Dxk
        \label{eq:consecutive-steps-consensus-bound}
    \end{align}
    
    where (a) follows from \cref{lemma:R-contraction} and (b) follows from $D(\boldx[k], \phi_k)\triangleq \sqrt{\sum_{i=1}^n \sum_{j=1}^n \phiki \phikj \normsq{x_i[k]-x_j[k]}} = \sqrt{2} \sqrt{\sum_{i=1}^n \phiki \normsq{x_i[k]-\hat x[k]}}$ which can be obtained by expanding the squared norm into the inner product and simplifying similar to \cite[eq. (7a)]{nedich2022ab-timevarying}.

    Now, we turn back to the first term in \cref{eq:consecutive-terms-split}. From the row-stochasticity of $R$ and the convexity of squared norm, we have
    \begin{align}
        \sqrt{\sum_{i=1}^n \normsq{\sum_{j=1}^n [R_k]_{ij} (z_j[k]-x_j[k])}}&\leq \sqrt{\frac{1}{\min \phi_{k+1}}} \sqrt{\sum_{i=1}^n \phikpi \sum_{j=1}^n [R_k]_{ij} \normsq{z_j[k]-x_j[k]}}\nonumber\\
        &\overset{(a)}{=}\sqrt{\frac{1}{\min \phi_{k+1}}}\sqrt{\sum_{j=1}^n \phikj\normsq{z_j[k]-x_j[k]}}\nonumber\\
        &\leq \sqrt{\frac{1}{\min \phi_{k+1}}} \sqrt{\sum_{j=1}^n \phikj\normsq{z_j[k]-w_j[k]}}+\sqrt{\frac{1}{\min \phi_{k+1}}}\sqrt{\sum_{j=1}^n \phikj \normsq{w_j[k]-x_j[k]}}\label{eq:zk-xk-diff},
    \end{align}
    where (a) follows from the fact $\phi_{k+1}^T R = \phi_k^T$ (\cref{lemma:phi-sequence}). Notice that the first term in \cref{eq:zk-xk-diff} is exactly equal to $\sqrt{\frac{1}{\min \phi_{k+1}}}\normphik{\boldz[k]-\boldw[k]}$ which can be bounded using \cref{proposition:imperfect_gradient_error}. 
     Then, we look at $\normsq{w_j[k]-x_j[k]}$ in \cref{eq:zk-xk-diff}. We have
    \begin{align*}
        \normsq{w_j[k]-x_j[k]}&=\normsq{(1-\lambda)x_j[k]+\lambda \proj{x_j[k]-\eta n\pikj \nabla f(x_j[k])}-x_j[k]}\\
        &\leq \lambda^2 \normsq{x_j[k]-\proj{x_j[k]-\eta n \pikj \nabla f(x_j[k])}}.
    \end{align*}
    We have
    \begin{align*}
        \sqrt{\sum_{j=1}^n \phikj \normsq{w_j[k]-x_j[k]}}&=\lambda \sqrt{\sum_{j=1}^n \phikj \normsq{x_j[k]-\proj{x_j[k]-\eta n\pikj \nabla f(x_j[k])}}}\\
        &\leq \lambda \sqrt{\sum_{j=1}^n \phikj \normsq{x_j[k]-x^*}} + \lambda \sqrt{\sum_{j=1}^n \phikj \normsq{x^*-\proj{x_j[k]-\eta n\pikj \nabla f(x_j[k])}}}\\
        &\leq \lambda (1+ q_{k}(\eta, 1))\normphik{\boldx[k]-\boldx^*}.
    \end{align*}
    Finally, using this and the bound for $\normphik{\boldz[k]-\boldw[k]}$ (\cref{proposition:imperfect_gradient_error}) to simplify \cref{eq:zk-xk-diff}, we get
    \begin{align*}
        \sqrt{\sum_{i=1}^n \normsq{\sum_{j=1}^n [R_k]_{ij}(z_j[k]-x_j[k])}}&\leq \lambda \sqrt{\frac{1}{\min \phi_{k+1}}}(1+q_k(\eta, 1))\normphik{\boldx[k]-\boldx^*}\\
        &+\eta \lambda L \sqrt{\frac{n}{\min \phi_{k+1}\min \phi_k}} \Dxk \\ &+ \eta \lambda \sqrt{\frac{1}{\min \phi_{k+1}}} \Syk
    \end{align*}
Finally, we combine with the bound in \cref{eq:consecutive-steps-consensus-bound} to get
\begin{align*}
    \norm{\boldx[k+1]-\boldx[k]}_{\boldone} \leq \lambda \sqrt{\frac{1}{\min \phi_{k+1}}} (1+q_{k}(\eta, 1)) \normphik{\boldx[k]-\boldx^*} \\
    + \left[\frac{1}{\sqrt{2}} \left( \sqrt{\frac{1}{\min \phi_k}} + \sigma_k \sqrt{\frac{1}{\min \phi_{k+1}}}\right)+ \eta \lambda L \sqrt{\frac{n}{\min \phi_{k+1}\min \phi_k}}\right] \Dxk\\
    +\eta \lambda \sqrt{\frac{1}{\min \phi_{k+1}}} \Syk.
\end{align*}
Letting $\varphi_k=\sqrt{\frac{1}{\min \phi_k}}$ and dropping the $\frac{1}{\sqrt 2 }$ term gives the desired result.

Now, we can establish the final bound for the gradient tracking error

\subsubsection{Proof of \cref{lemma:gradient-tracking}}

    Follows from plugging in the result of \cref{proposition:consecutive-terms-bound} in place of the term $\norm{\boldx[k+1]-\boldx[k]}_{\boldone}$ in \cref{proposition:grad-tracking-bound-with-consecutive-terms}.
\subsection{Proof of \Cref{theorem:convergence-of-protocol}}\
\label{sec:proof_main_thm}
It is enough to show that the diagonal entries of $M(\eta, \lambda)$ are less than 1 and $\det M(\eta, \lambda) - I < 0$. Hence, for the diagonal entries, we need $\lambda < \frac{1-\sigma}{2 \varphi \sqrt{n}}$ and $\lambda < \frac{1-\tau}{r \varphi}$. Then, the determinant is as follows
\begin{align*}
    \det M(\eta, \lambda) - I &= \begin{vmatrix}
        -\eta \lambda n \psi \mu & \lambda \varphi \sqrt{n} & \lambda  L^{-1}  \\
            2\lambda & \sigma-1 + 2\lambda \sqrt{n} \varphi & 2 \lambda  L^{-1} \\
            2 \lambda Lr \varphi & Lr\varphi (1+\sigma) + \lambda L r \varphi^2 \sqrt{n} & \tau-1 + \lambda r \varphi 
    \end{vmatrix}
    \\
            &= -\eta \lambda n \psi \mu \begin{vmatrix}
                \sigma - 1 + 2\lambda \sqrt{n} \varphi & 2 \lambda L^{-1}\\
                Lr \varphi(1+\sigma) + \lambda Lr\varphi^2 \sqrt{n} & \tau - 1 + \lambda r \varphi
            \end{vmatrix}\\
            &- \lambda \varphi \sqrt{n} \begin{vmatrix}
                2 \lambda & 2 \lambda L^{-1}\\
                2 \lambda Lr\varphi & \tau - 1+ \lambda r \varphi
            \end{vmatrix}\\
            &+ \lambda L^{-1} \begin{vmatrix}
                2 \lambda & \sigma -1 + 2\lambda \sqrt{n} \varphi \\
                2 \lambda Lr\varphi & Lr\varphi(1+\sigma) +\lambda Lr \varphi^2 \sqrt{n}
            \end{vmatrix}
\end{align*}
Hence, we perform the following computations
\begin{align*}
    &-\eta \lambda n \psi \mu \begin{vmatrix}
                \sigma - 1 + 2\lambda \sqrt{n} \varphi & 2 \lambda L^{-1}\\
                Lr \varphi(1+\sigma) + \lambda Lr\varphi^2 \sqrt{n} & \tau - 1 + \lambda r \varphi
            \end{vmatrix}\\
            &= -\eta \lambda n\psi \mu \left[(\sigma-1)(\tau-1)+\lambda(2\varphi \sqrt{n}(\tau-1)+r\varphi(\sigma-1))+2\lambda^2 r\varphi^2 \sqrt{n} -2\lambda r\varphi (1+\sigma) - 2\lambda^2 r\varphi^2 \sqrt{n} \right]\\
            &=- \eta \lambda n\psi \mu \left[(\sigma-1)(\tau-1)-\lambda (2 \varphi \sqrt{n} (1-\tau)+r\varphi (1-\sigma) + 2r \varphi(1+\sigma) )\right]\\
            &=-\lambda \left[ \eta n\psi \mu (1-\sigma)(1-\tau) - \lambda \eta n\psi \mu (2 \varphi \sqrt{n} (1-\tau)+r\varphi (1-\sigma) + 2r \varphi(1+\sigma) )\right].
\end{align*}
The second term is
\begin{align*}
    &- \lambda \varphi \sqrt{n} \begin{vmatrix}
                2 \lambda & 2 \lambda L^{-1}\\
                2 \lambda Lr\varphi & \tau - 1+ \lambda r \varphi
                \end{vmatrix}\\
    &=-\lambda \varphi\sqrt{n} \left[2\lambda(\tau-1)+ 2\lambda^2 r\varphi- 4\lambda^2 r\varphi\right]\\
    &=-\lambda \left[-2\lambda \varphi \sqrt{n} (1-\tau)-2\lambda^2 r\varphi\right].
\end{align*}
And the final term is
\begin{align*}
    &\lambda L^{-1} \begin{vmatrix}
                2 \lambda & \sigma -1 + 2\lambda \sqrt{n} \varphi \\
                2 \lambda Lr\varphi & Lr\varphi(1+\sigma) +\lambda Lr \varphi^2 \sqrt{n}
            \end{vmatrix}\\
     &=\lambda L^{-1}\left[2\lambda L r\varphi(1+\sigma) + 2\lambda^2 Lr \varphi^2 \sqrt{n} - 2\lambda Lr\varphi (\sigma-1) - 4\lambda^2 Lr\varphi^2 \sqrt{n}\right]\\
     &=-\lambda \left[-2\lambda( r\varphi(1+\sigma) + r\varphi (1-\sigma))+2\lambda^2 r\varphi^2 \sqrt{n}\right].
\end{align*}
Combining all the terms, we get
\begin{align*}
    \det M(\eta, \lambda)-I&= -\lambda\left[\eta n \psi \mu (1-\sigma)(1-\tau) - \lambda K +2\lambda^2r\varphi(\sqrt{n}\varphi - 1) \right],
\end{align*}
where $K=(1+\eta n \psi \mu) [2\varphi \sqrt{n}(1-\tau)+r\varphi(1-\sigma) +2r \varphi(1+\sigma)] > 0$ since $\tau < 1$ and $\sigma < 1$. Now, we inspect the $\lambda^2$ term in the parenthesis. Notice that $\varphi \geq 1$ since $\phiki \leq 1$. Since we also have $\sqrt{n} \geq 1$, we get $2\lambda^2 r\varphi(\varphi \sqrt{n}-1) \geq 0$. So,
\begin{align*}
    \det M(\eta, \lambda) - I \leq -\lambda (\eta n \psi \mu (1-\sigma) (1-\tau) - \lambda K).
\end{align*}
Thus, to ensure that the determinant is negative, it is sufficient to choose $\lambda$ such that $\lambda K < \eta n \psi \mu (1-\sigma) (1-\tau)$ which gives is $\lambda < \frac{\eta n \psi \mu (1-\sigma) (1-\tau)}{K}$. Combining all the bounds, we have
\begin{align*}
    \lambda < \min \left\{\frac{1-\sigma}{2\varphi \sqrt{n}}, \frac{1-\tau}{r\varphi}, \frac{\eta n \psi \mu (1-\sigma)(1-\tau)}{(1+\eta n \psi \mu) [2\varphi \sqrt{n}(1-\tau)+r\varphi(1-\sigma) +2r \varphi(1+\sigma)]}\right\}
\end{align*}
as desired.
\end{document}